\documentclass[12pt]{amsart}
\usepackage{amstext}
\usepackage{amsthm}
\usepackage{amsmath}
\usepackage{amssymb}
\usepackage{latexsym}
\usepackage{amsfonts}
\usepackage{graphicx}
\usepackage{texdraw}
\usepackage{graphpap}
\usepackage[pagebackref,hypertexnames=false, colorlinks, citecolor=red, linkcolor=red]{hyperref}
\usepackage[backrefs]{amsrefs}

\input txdtools
\begin{document}
\title{On the Uniqueness Sets in the Fock Space}

\begin{abstract}
It was known to von Neumann in the 1950's that the integer lattice $\mathbb{Z}^2$ forms a uniqueness set for the Bargmann-Fock space.  It was later demonstrated by Seip and Wallst\'en that a sequence of points $\Gamma$ that is uniformly close to the integer lattice is still a uniqueness set.  We show in this paper that the uniqueness sets for the Fock space are preserved under much more general perturbations.
\end{abstract}

\author[M. Mitkovski]{Mishko Mitkovski$^\dagger$}
\address{Mishko Mitkovski, Department of Mathematical Sciences\\ Clemson University\\ O-110 Martin Hall, Box 340975\\ Clemson, SC USA 29634}
\email{mmitkov@clemson.edu}
\urladdr{http://people.clemson.edu/~mmitkov/}
\thanks{$\dagger$ Research supported in part by National Science Foundation DMS grant \# 1101251.}

\author[B. D. Wick]{Brett D. Wick$^\ddagger$}
\address{Brett D. Wick, School of Mathematics\\ Georgia Institute of Technology\\ 686 Cherry Street\\ Atlanta, GA USA 30332-0160}
\email{wick@math.gatech.edu}
\urladdr{www.math.gatech.edu/~wick}
\thanks{$\ddagger$ Research supported in part by National Science Foundation DMS grants \# 1001098 and \# 955432.}

\maketitle

%Commands Used%

\newcommand{\norm}[1]{\ensuremath{\left\|#1\right\|}}
\newcommand{\abs}[1]{\ensuremath{\left\vert#1\right\vert}}
\newcommand{\p}{\ensuremath{\partial}}
\newcommand{\pr}{\mathcal{P}}

\newcommand{\pbar}{\ensuremath{\bar{\partial}}}
\newcommand{\db}{\overline\partial}
\newcommand{\B}{\mathbb{B}}
\newcommand{\Sp}{\mathbb{S}}
\newcommand{\scrH}{\mathcal{H}}
\newcommand{\scrL}{\mathcal{L}}
\newcommand{\td}{\widetilde\Delta}

\newcommand{\La}{\langle }
\newcommand{\Ra}{\rangle }
\newcommand{\rk}{\operatorname{rk}}
\newcommand{\card}{\operatorname{card}}
\newcommand{\ran}{\operatorname{Ran}}
\newcommand{\osc}{\operatorname{OSC}}
\newcommand{\im}{\operatorname{Im}}
\newcommand{\re}{\operatorname{Re}}
\newcommand{\tr}{\operatorname{tr}}
\newcommand{\supp}{\operatorname{supp}}
\newcommand{\vf}{\varphi}
\newcommand{\f}[2]{\ensuremath{\frac{#1}{#2}}}

\def\to{\rightarrow}
\def\C{{\mathbb{C}}}
\def\D{{\mathbb{D}}}
\def\Z{{\mathbb{Z}}}
\def\N{{\mathbb{N}}}
\def\R{{\mathbb R}}
\def\T{{\mathbb T}}
\def\H{{\mathbb H}}
\def\L{{\Lambda}}
\def\G{{\Gamma}}

\def\AA{{\mathcal{A}}}
\def\BB{{\mathcal B}}
\def\CC{{\mathcal C}}
\def\DD{{\mathcal{D}}}
\def\EE{{\mathcal E}}
\def\FF{{\mathcal F}}
\def\GG{{\mathcal{G}}}
\def\HH{{\mathcal H}}
\def\JJ{{\mathcal{J}}}
\def\LL{{\mathcal{L}}}
\def\MM{{\mathcal M}}
\def\NN{{\mathcal N}}
\def\PP{{\mathcal{P}}}
\def\QQ{{\mathcal{Q}}}
\def\RR{\mathcal{R}}
\def\SS{{\mathcal{S}}}
\def\TT{{\mathcal{T}}}
\def\UU{{\mathcal{U}}}
\def\XX{{\mathcal{X}}}
\def\YY{{\mathcal{Y}}}
\def\ZZ{{\mathcal{Z}}}

\def\lmn{{\lambda_{mn}}}
\def\gmn{{\gamma_{mn}}}

\def\a{{\alpha}}
\def\e{{\epsilon}}
\def\d{{\delta}}
\def\D{{\Delta}}

\def\o{{\omega}}
\def\sL{{\sigma_{\Lambda}}}
\def\sG{{\sigma_{\Gamma}}}
\def\sump{{{\sum}'}}

\numberwithin{equation}{section}

\newtheorem{thm}{Theorem}[section] \newtheorem{lm}[thm]{Lemma} %
\newtheorem{cor}[thm]{Corollary} \newtheorem{conj}[thm]{Conjecture} %
\newtheorem{prob}[thm]{Problem} \newtheorem{prop}[thm]{Proposition} %
\newtheorem*{prop*}{Proposition}

\newtheorem{define}[thm]{Definition}

\theoremstyle{remark} \newtheorem{rem}[thm]{Remark} %
\newtheorem*{rem*}{Remark} \newtheorem{quest}[thm]{Question}

\section{Introduction and Statement of Main Results}

For a given sequence $\G=\left\{\gmn\right\}$ in $\C$ and a window function $g\in L^2(\R)$ a Gabor system $\GG(\G,g)$ is defined as
$$\GG(\G,g)=\left\{e^{i\textnormal{Re}{\gmn}t}g(t-\textnormal{Im}{\gmn}):\gmn\in\G\right\}.$$ 
Gabor introduced these systems in 1946~\cite{Gab}  for applications in signal processing, and they have been widely used since then. However, one of the most fundamental questions concerning Gabor systems which asks for the description of all time-frequency sequences $\G$ such that $\GG(\G,g)$ is complete in $L^2(\R)$ remained widely open. A complete description is known only in the case of a Gaussian window $g(t)=e^{-\pi t^2}$ and $\G$ being a lattice. In this classical case, it was von~Neumann~\cite{vonNeu} who first observed (without proof) that the system $\GG(\G,g)$ is complete when $\G$ is the integer lattice. Latter many different proofs were given~\cite{BGZ, BBGK, Per, SW} treating also the case of $\G$ being an arbitrary lattice. However, the case  of irregular systems ($\G$ is not a lattice) remains a complete mystery to date, even in this classical case. It is worth mentioning that more recently a complete description of irregular Gabor frames was given by Seip and Wallsten~\cite{SW,Seip}. 

What makes the case of a Gaussian window $g(t)=e^{-\pi t^2}$ simpler to handle, is the fact that in its treatment one can use complex analysis.
Namely, the Bargmann transform $$ Bf(z)=\sqrt[4]{2}\int f(t)e^{2\pi it z-\pi t^2-\frac{\pi}{2}z^2}dt$$ defines a unitary map between $L^2(\R)$ and the  
Bargmann-Fock space $\FF$  consisting of all entire functions such that 
$$\int_{\C}|F(z)|^2e^{-\pi|z|^2}dA(z)<\infty.$$ Therefore, most of the questions concerning the basis properties of Gabor systems can be translated into questions about the Bargmann-Fock space $\FF$ and can be treated with the tools from complex analysis. In particular, the completeness problem for irregular Gabor systems is equivalent to the uniqueness set problem in the Bargmann-Fock space $\FF$. Recall that  a sequence $\G=\{\gmn\}$ in $\C$ is said to be a uniqueness sequence for $\FF$ if every function $F(z)\in\FF$ that vanishes identically on $\G$ must be identically zero. As described in~\cite{Per}, classical results from the theory of entire functions can be used in many cases to check whether a sequence is a uniqueness set. Sequences that were not treated in~\cite{Per} are the ones satisfying the following three properties:

\begin{itemize}

\item[(a)]  $\sum \frac{1}{|\gmn|^{2+\e}}$ converges for $\e>0$ and diverges for $\e\leq 0$
\item[(b)] The classical upper density $D^{+}(\G)$ of $\G$ satisfies $0<D^{+}(\G)<\infty$
\item[(c)] $\abs{\sum_{|\gmn|<R}\frac{1}{\gmn^2}}$ is bounded as a function of $R$.

\end{itemize}
In this paper we consider a class of sequences which satisfy these three properties and characterize among them the uniqueness sequences in $\FF$. The hope is that similar sequences can be used to define the right density notion which will give a complete description of the uniqueness sets in $\FF$. 

\subsection{Main Results}

Denote by $\L:=\{\lmn\}$ the integer lattice, i.e., $\lmn=m+in, \hspace{0.2cm} m,n\in\Z.$  We introduce a more general class of sequences that we are interested in studying

\begin{define}
A sequence $\G = \{\gmn\}$ of complex numbers is $d-$regular if there is a constant $ c >0$ and a function $\phi: [1,\infty) \to \R$ with $\frac{\phi(t)}{t}$ non-increasing and satisfying $\int_1^{\infty}\frac{\phi(t)}{t^2}dt<\infty,$ such that
\begin{itemize}
\item[(a)] $\abs{\lambda_{mn}-d\gamma_{mn}} \leq \phi( \abs{\lambda_{mn}})$
\item[(b)] $\abs{\gamma_{mn}-\gamma_{m'n'}} \geq c \abs{\lambda_{mn}-\lambda_{m'n'}}$.
\end{itemize}
\end{define}

Our main result is the following new result providing information about the uniqueness sets for the Fock space.

\begin{thm}\label{mainfornow} Let $\G$ be a $d-$regular sequence.
\begin{itemize}
\item[(a)] If $d<1$ then $\G$ is not a uniqueness set for the Bargmann-Fock space $\FF$.
\item[(b)] If $d>1$ then $\G$ is a uniqueness set for the Bargmann-Fock space $\FF$.
\end{itemize}  
\end{thm}

\begin{rem} The case $d=1$ is inconclusive. Namely, the integer lattice $\L$ is a 1-regular sequence which is a uniqueness set. Removing two points from $\L$ gives a 1-regular sequence which is \emph{not} a uniqueness set. 
\end{rem}

Throughout the paper, the constants can change from line to line, and depends on the appropriate parameters in the estimates in question.

\section{Important Lemmas}

In the following lemma we give some properties of $d$-regular sequences which will be useful in the estimates that follow.
\begin{lm}\label{dreg} If $\G = \{\gmn\}$ is $d$-regular, then
\begin{itemize}
\item[(a)]  $ \frac{1}{\kappa} \leq \frac{\abs{\lambda_{mn}}}{d\abs{\gamma_{mn}}}\leq \kappa$ for some $\kappa> 1,$
\item[(b)] $\sideset{}{'}{\sum}_{m,n}\frac{1}{\abs{\lambda_{mn}}}\abs{\frac{d}{\lambda_{mn}}-\frac{1}{\gamma_{mn}}}<\infty,$
\item[(c)]  $\sideset{}{'}{\sum}_{\abs{\lmn}\leq |z|}\abs{\frac{d}{\lmn}-\frac{1}{\gmn}}=o(\abs{z}).$
\end{itemize}
Here $\sideset{}{'}{\sum}$ denotes the sum taken over the set in question, but with $\lambda_{mn}\neq 0$.
\end{lm}

\begin{proof}
Part (a) follows from the definition of $d$-regular sequence which in fact implies that $\lim_{m,n \to \infty} \frac{\lmn}{\gmn} =d$. 

For part (b) first notice that our condition on $\phi(t)$ says that $$\int_{1}^{\infty}\frac{\phi(t)}{t^2}dt<\infty.$$
Using this we obtain
$$\sideset{}{'}{\sum}_{m,n}\frac{1}{\abs{\lambda_{mn}}}\abs{\frac{d}{\lambda_{mn}}-\frac{1}{\gamma_{mn}}}\leq \kappa d \sideset{}{'}{\sum}_{m,n}\frac{\abs{\lambda_{mn}-d\gamma_{mn}}}{\abs{\lambda_{mn}}^3}\leq \kappa d \sideset{}{'}{\sum}_{m,n}\frac{\phi( \abs{\lambda_{mn}})}{\abs{\lambda_{mn}}^3}\leq O(1)\int_{1}^{\infty}\frac{\phi(t)}{t^2}dt<\infty.$$

Finally, we show that (c) is a consequence of (b).
Let $\e>0$. There exists $w\in\C$ such that $$\sum_{|\lmn|>|w|}\frac{1}{|\lmn|}\abs{\frac{d}{\lmn}-\frac{1}{\gmn}}<\e.$$ Therefore, $$\sum_{|w|\leq |\lmn|\leq |z|}\abs{\frac{d}{\lmn}-\frac{1}{\gmn}}<\sum_{|w|\leq |\lmn|\leq |z|}\frac{|z|}{|\lmn|}\abs{\frac{d}{\lmn}-\frac{1}{\gmn}}\leq \e|z|.$$ Now, 

$$\frac{1}{|z|}\sideset{}{'}{\sum}_{|\lmn|\leq |z|}\abs{\frac{d}{\lmn}-\frac{1}{\gmn}}=\frac{1}{|z|}\sideset{}{'}{\sum}_{|\lmn|< |w|}\abs{\frac{d}{\lmn}-\frac{1}{\gmn}}+\frac{1}{|z|}\sum_{{|w|\leq |\lmn|\leq |z|}}\abs{\frac{d}{\lmn}-\frac{1}{\gmn}}.$$ By fixing $w$ and letting $|z|\to\infty$ we obtain $$\lim_{|z|\to\infty}\frac{1}{|z|}\sideset{}{'}{\sum}_{|\lmn|\leq |z|}\abs{\frac{d}{\lmn}-\frac{1}{\gmn}}\leq \e.$$ Since $\e$ was arbitrary we are done. 
\end{proof}

\begin{lm}\label{phi_est} If $\frac{\phi(t)}{t}$ is a non-increasing positive function such that  $\int_1^{\infty}\frac{\phi(t)}{t^2}dt<\infty$ then $\phi(t)=o\left(\frac{t}{\ln{t}}\right)$ as $t\to\infty$.
\end{lm}

\begin{proof} If $\phi(t)\neq o\left(\frac{t}{\ln{t}}\right)$ then there exists $\e>0$ and a sequence $(b_k)$ (with terms $>1$) such that $\phi(b_k)\ln{b_k}>\e b_k$. Moreover, for a fixed small $\d>0$ we can choose this sequence so that $b_k<(b_{k+1})^{1-\d}$. Denote $a_k=(b_{k})^{1-\d}$. We obtain the following contradiction 
$$ \int_1^{\infty}\frac{\phi(t)}{t^2}dt\geq \sum_k \int_{a_k}^{b_k}\frac{\phi(t)}{t^2}dt \geq \sum_k\frac{\phi(b_k)}{b_k}\int_{a_k}^{b_k}\frac{dt}{t}>\d \sum_k \e=\infty.$$

\end{proof}

Let $\sL(z)$ be the usual Weierstrass $\sigma-$function (from the theory of elliptic functions), defined by  $$\sigma_{\L}(z)=z{\prod_{m,n}}'\left(1-\frac{z}{\lmn}\right)\exp\left(\frac{z}{\lmn}+\frac{z^2}{2\lmn}\right).$$
The fundamental property that $\sigma_{\L}(z)$ possesses is the quasi-periodicity: 
$$
\sL \left(z+\lmn\right)=-\sL (z)\exp\left(\left(m\eta_1+n\eta_2\right)\left(z+\frac{\lmn}{2}\right)\right),
$$ where $\eta_1=2\frac{\sL'(1/2)}{\sL(1/2)}$, $\eta_2=2\frac{\sL'(i/2)}{\sL(i/2)}$. An immediate consequence of the quasi-periodicity of $\sL(z)$ is the following inequality:
\begin{equation}
\label{Weier-Est} 
C_1 \textnormal{dist}(z,\L)\leq |\sL(z)e^{-\frac{\pi}{2}|z|^2}|\leq C_2 .
\end{equation}  

To each $d-$regular sequence $\G:=\{\gmn\}$ we associate the entire function $\sG(z)$ by 
$$
\sigma_{\Gamma}(z):=\left(z-\gamma_{00}\right)\sideset{}{'}{\prod}_{m,n}\left(1-\frac{z}{\gamma_{mn}}\right)\exp\left(\frac{z}{\gamma_{mn}}+\frac{1}{2}\frac{d^2 z^2}{\lambda_{mn}^2}\right).
$$
In the case when $\G=\frac{1}{d}\L$ is the scaled integer lattice $\sG(z)$ is just the Weierstrass sigma-function $\sL(dz)$. Since $\sG(z)$ is not the usual Weierstrass product, its convergence must be verified.

\begin{lm}
The function $\sigma_{\Gamma}(z)$ is a well defined entire function.
\end{lm}
\begin{proof}
Fix $R>0$. For $|z|\leq R$ and $|\gmn|>2R$ we have
\begin{eqnarray*}\abs{\log\abs{1-\frac{z}{\gmn}}+\textnormal{Re}\,\frac{z}{\gmn}+\textnormal{Re}\,\frac{(dz)^2}{2\lmn^2}}&=& \abs{-\textnormal{Re}\sum_{k\geq 2}\frac{z^k}{k\gmn^k}+\textnormal{Re}\,\frac{(dz)^2}{2\lmn^2}} \\
&\leq&\frac{|z|^2}{2}\abs{\frac{d}{\lmn}-\frac{1}{\gmn}}^2+\frac{|z|^2}{\abs{\gmn}}\abs{\frac{d}{\lmn}-\frac{1}{\gmn}}+\sum_{k\geq 3}\frac{|z|^k}{k|\gmn|^k} \\
&\leq&\frac{|z|^2}{2}\abs{\frac{d}{\lmn}-\frac{1}{\gmn}}^2+\frac{|z|}{2}\abs{\frac{d}{\lmn}-\frac{1}{\gmn}}+\sum_{k\geq 3}\frac{|z|^k}{k|\gmn|^k} \\
&\leq& \frac{|z|^2}{2}\abs{\frac{d}{\lmn}-\frac{1}{\gmn}}^2+\frac{d|z|}{2|\lmn|}\abs{\frac{d}{\lmn}-\frac{1}{\gmn}}+\frac{|z|^3}{|\gmn|^3}O(1).\end{eqnarray*}
 Now, summation over all $|\gmn|>2R$ gives easily an upper bound $O(R^3)$ which is sufficient to deduce the uniform  convergence on compact sets. So, $\sG(z)$ is a well defined entire function.
\end{proof}

Our goal in the next lemma is to obtain analogous estimates as in \eqref{Weier-Est} for the function $\sigma_{\Gamma}(z)$.  This lemma is analogous to a related fact appearing in \cite{SW}.

\begin{lm}\label{mainlemma}
Let $\Gamma$ be a $d$-regular sequence.  Then there exist constants $C_1$ and $C_2$ such that
\begin{equation}
\label{Seip-Est}
C_1 \textnormal{dist}(z,\Gamma) e^{-o(\abs{z}^2)}\leq\abs{\sigma_{\Gamma}(z) e^{-\frac{d^2\pi}{2}\abs{z}^2}}\leq C_2 e^{o(\abs{z}^2)}.
\end{equation}
Here the constants $C_1$ and $C_2$ depend on $c$ and  $\phi(t)$ from the definition of the $d$-regular sequence.
\end{lm}

\begin{proof}
For the proof, we define an auxiliary function
$$
h(z):=\frac{\sigma_{\Gamma}(z)}{\sigma_{\Lambda}(d z)}\frac{\textnormal{dist}(d z,\Lambda)}{\textnormal{dist}(z,\Gamma)}
$$
Note that this gives us the representation
$$
\sigma_{\Gamma}(z)=\sigma_{\Lambda}(d z)\frac{\textnormal{dist}(z,\Gamma)}{\textnormal{dist}(d z,\Lambda)}h(z)
$$
The plan now is to show that
\begin{equation}
\label{goalest}
C_1e^{-o(\abs{z}^2)}\leq\abs{h(z)}\leq C_2\frac{\textnormal{dist}(d z,\Lambda)}{\textnormal{dist}(z,\Gamma)}e^{o(\abs{z}^2)}.
\end{equation}

If we have estimate \eqref{goalest}, then by coupling this with \eqref{Weier-Est} we see that 
\begin{eqnarray*}
\abs{\sigma_\Gamma(z) e^{-\frac{d ^2\pi}{2}\abs{z}^2}} & = & \frac{\textnormal{dist}(z,\Gamma)}{\textnormal{dist}(dz,\Lambda)}\abs{h(z)}\abs{\sigma_{\Lambda}(d z)e^{-\frac{d^2\pi}{2}\abs{z}^2}}\\
& \leq & C_2 \frac{\textnormal{dist}(z,\Gamma)}{\textnormal{dist}(d z,\Lambda)}\frac{\textnormal{dist}(d z,\Lambda)}{\textnormal{dist}(z,\Gamma)} e^{o(\abs{z}^2)}\abs{\sigma_{\Lambda}(d z)e^{-\frac{d ^2\pi}{2}\abs{z}^2}}\\
& = & C_2 e^{o(\abs{z}^2)}\abs{\sigma_{\Lambda}(d z)e^{-\frac{d ^2\pi}{2}\abs{z}^2}}\\
& \leq & C_2' e^{o(\abs{z}^2)}.
\end{eqnarray*}
Where in the estimates above, we first used the upper estimate in \eqref{goalest} and the upper estimate in \eqref{Weier-Est}.  This estimate then gives us that the upper estimate in \eqref{Seip-Est} holds.

Similarly, first using the lower estimate in \eqref{goalest}, and then the lower estimate in \eqref{Weier-Est} yields the following
\begin{eqnarray*}
\abs{\sigma_\Gamma(z)e^{-\frac{d ^2\pi}{2}\abs{z}^2}} & = & \frac{\textnormal{dist}(z,\Gamma)}{\textnormal{dist}(d z,\Lambda)}\abs{h(z)}\abs{\sigma_{\Lambda}(d z)e^{-\frac{d ^2\pi}{2}\abs{z}^2}}\\
& \geq & C_1 \frac{\textnormal{dist}(z,\Gamma)}{\textnormal{dist}(d z,\Lambda)}e^{-o(\abs{z}^2)}\abs{\sigma_{\Lambda}(d z)e^{-\frac{d ^2\pi}{2}\abs{z}^2}}\\
& \geq & C_1' e^{-o(\abs{z}^2)}\textnormal{dist}(z,\Gamma)\frac{\textnormal{dist}(d z,\Lambda)}{\textnormal{dist}(d z,\Lambda)}\\
& = & C_1' e^{-o(\abs{z}^2)} \textnormal{dist}(z,\Gamma).
\end{eqnarray*}
which is the lower estimate in \eqref{Seip-Est}.

Working with the definition of $h(z)$ we see that we can write it in the following way:
\begin{eqnarray*}
h(z) & = & \frac{\sigma_{\Gamma}(z)}{\sigma_{\Lambda}(d z)}\frac{\textnormal{dist}(d z,\Lambda)}{\textnormal{dist}(z,\Gamma)}\\
 & = & \frac{\textnormal{dist}(d z,\Lambda)}{\textnormal{dist}(z,\Gamma)}\frac{(z-\gamma_{00})}{d z}\sideset{}{'}{\prod}_{m,n}\frac{\left(1-\frac{z}{\gamma_{mn}}\right)}{\left(1-\frac{d z}{\lambda_{mn}}\right)}\exp\left(\frac{z}{\gamma_{mn}}-\frac{d z}{\lambda_{mn}}\right)\\
 & = & \frac{\textnormal{dist}(d z,\Lambda)}{\textnormal{dist}(z,\Gamma)}\frac{(z-\gamma_{00})}{d z} \exp\left(z\sideset{}{'}{\sum}_{m,n}\left(\frac{1}{\gamma_{mn}}-\frac{d}{\lambda_{mn}}\right)\right) \sideset{}{'}{\prod}_{m,n}\frac{\left(1-\frac{z}{\gamma_{mn}}\right)}{\left(1-\frac{d z}{\lambda_{mn}}\right)}.
\end{eqnarray*}

We now factorize the function $h(z)$ into pieces based on this representation.  Fix some $\epsilon>0$.  Write $h(z)=h_1(z)h_2(z)h_3(z)$ where
\begin{eqnarray*}
h_1(z) & = & \exp\left(z\sideset{}{'}{\sum}_{\abs{\lambda_{mn}}\leq \kappa^{1+\epsilon}d\abs{z}}\left(\frac{1}{\gamma_{mn}}-\frac{d}{\lambda_{mn}}\right)\right)\\
h_2(z) & = & \frac{\textnormal{dist}(d z,\Lambda)}{\textnormal{dist}(z,\Gamma)}\frac{(z-\gamma_{00})}{d z}\sideset{}{'}{\prod}_{\abs{\lambda_{mn}}\leq \kappa^{1+\epsilon}d\abs{z}}\frac{\left(1-\frac{z}{\gamma_{mn}}\right)}{\left(1-\frac{d z}{\lambda_{mn}}\right)}\\
h_3(z) & = & \exp\left(z\sum_{\abs{\lambda_{mn}}>\kappa^{1+\epsilon}d\abs{z}}\left(\frac{1}{\gamma_{mn}}-\frac{d}{\lambda_{mn}}\right)\right)\prod_{\abs{\lambda_{mn}}> \kappa^{1+\epsilon}d\abs{z}}\frac{\left(1-\frac{z}{\gamma_{mn}}\right)}{\left(1-\frac{d z}{\lambda_{mn}}\right)}.
\end{eqnarray*}
We will show that the following estimates hold
\begin{eqnarray}
& e^{-o(\abs{z}^2)}\leq\abs{h_1(z)}\leq e^{o(\abs{z}^2)} & \label{h1est}\\
& C_1e^{-o(\abs{z}^2)}\leq\abs{h_2(z)}\leq C_2\frac{\textnormal{dist}(d z,\Lambda)}{\textnormal{dist}(z,\Gamma)}e^{o(\abs{z}^2)} & \label{h2est}\\
& e^{-o(\abs{z}^2)}\leq\abs{h_3(z)}\leq e^{o(\abs{z}^2)}. & \label{h3est}
\end{eqnarray}
Combining estimates \eqref{h1est}, \eqref{h2est} and \eqref{h3est} we see that \eqref{goalest} holds.

We first consider the function $h_1(z)$.  Note that
\begin{eqnarray*}
\abs{h_1(z)} & = & \exp\left(\textnormal{Re}\left(\sideset{}{'}{\sum}_{\abs{\lambda_{mn}}\leq \kappa^{1+\epsilon}d\abs{z}}\left(\frac{z}{\gamma_{mn}}-\frac{dz}{\lambda_{mn}}\right)\right)\right).
\end{eqnarray*}
Since for any complex number $-\abs{w}\leq \textnormal{Re}\, w\leq\abs{w}$ we have that
\begin{eqnarray*}
-\abs{z}\sideset{}{'}{\sum}_{\abs{\lambda_{mn}}\leq \kappa^{1+\epsilon}d\abs{z}}\abs{\frac{d}{\lambda_{mn}}-\frac{1}{\gamma_{mn}}}\leq \sideset{}{'}{\sum}_{\abs{\lambda_{mn}}\leq \kappa^{1+\epsilon}d\abs{z}} \textnormal{Re}\left(\frac{z}{\gamma_{mn}}-\frac{d z}{\lambda_{mn}}\right)\leq \abs{z}\sideset{}{'}{\sum}_{\abs{\lambda_{mn}}\leq \kappa^{1+\epsilon}d\abs{z}}\abs{\frac{d}{\lambda_{mn}}-\frac{1}{\gamma_{mn}}}.
\end{eqnarray*}
All together this implies 
$$
\exp\left(-\abs{z}\sideset{}{'}{\sum}_{\abs{\lambda_{mn}}\leq \kappa^{1+\epsilon}d\abs{z}}\abs{\frac{d}{\lambda_{mn}}-\frac{1}{\gamma_{mn}}}\right)\leq \abs{h_1(z)}\leq \exp\left(\abs{z}\sideset{}{'}{\sum}_{\abs{\lambda_{mn}}\leq \kappa^{1+\epsilon}d\abs{z}}\abs{\frac{d}{\lambda_{mn}}-\frac{1}{\gamma_{mn}}}\right).
$$
The desired estimate on $h_1(z)$ now follows from part (b) in Lemma~\ref{dreg}.

Next consider term $h_3(z)$.  First, take the absolute value of the expression defining $h_3(z)$ and then the logarithm to find that
$$
\log\abs{h_3(z)}=\sum_{\abs{\lambda_{mn}}>\kappa^{1+\epsilon}d\abs{z}}\textnormal{Re}\left(\frac{z}{\gamma_{mn}}-\frac{d z}{\lambda_{mn}}\right)+\log\abs{1-\frac{z}{\gamma_{mn}}}-\log\abs{1-\frac{d z}{\lambda_{mn}}}.
$$
We now will expand the logarithms via their power series representation, and collect terms.  Note that we have the power series representation for $\log(1-w)$ given by
$$
-\log(1-w)=\sum_{k=0}^\infty \frac{1}{k+1}w^{k+1}
$$
and so taking the real part of this series we have
$$
\log\abs{1-w}=-\sum_{k=0}^\infty\frac{1}{k+1}\textnormal{Re}\,(w^{k+1}). 
$$
Note that since $\kappa^{1+\epsilon}>1$ and since we are summing over the set $\abs{\lambda_{mn}}>\kappa^{1+\epsilon}d\abs{z}$, we then have that
$$
\frac{d\abs{z}}{\abs{\lambda_{mn}}}\leq\frac{1}{\kappa^{1+\epsilon}}<1.
$$
Similarly, since we suppose that $\abs{\lambda_{mn}}\leq\kappa d\abs{\gamma_{mn}}$ we have that
$$
\frac{\abs{z}}{\abs{\gamma_{mn}}}\leq \frac{d\kappa\abs{z}}{\abs{\lambda_{mn}}}\leq\frac{\kappa}{\kappa^{1+\epsilon}}=\frac{1}{\kappa^{\epsilon}}<1.
$$
Using the power series representation for fixed $\lambda_{mn}$ that satisfies $\abs{\lambda_{mn}}>\kappa^{1+\epsilon}d\abs{z}$ we find that
\begin{equation}
\label{partialh3}
\log\abs{h_3(z)} = \sum_{\abs{\lambda_{mn}}>\kappa^{1+\epsilon}d\abs{z}}\sum_{k=1}^\infty\frac{1}{k+1}\textnormal{Re}\left(\left(\frac{d z}{\lambda_{mn}}\right)^{k+1}-\left(\frac{z}{\gamma_{mn}}\right)^{k+1}\right).
\end{equation}

We continue the estimate of $h_3(z)$ in the following way:
\begin{eqnarray*}
\abs{\log\abs{h_3(z)}} & \leq & \sum_{\abs{\lambda_{mn}}>\kappa^{1+\epsilon}d\abs{z}}\sum_{k=1}^\infty\frac{\abs{z}^{k+1}}{k+1}\abs{\left(\frac{d }{\lambda_{mn}}\right)^{k+1}-\left(\frac{1}{\gamma_{mn}}\right)^{k+1}}\\
& = & \sum_{\abs{\lambda_{mn}}>\kappa^{1+\epsilon}d\abs{z}}\sum_{k=1}^\infty \frac{\abs{z}^{k+1}}{k+1}\abs{ \left( \frac{d }{\lambda_{mn}}- \frac{1}{\gamma_{mn}} \right) \sum_{j=0}^k \left( \frac{d }{\lambda_{mn}}\right)^j \left(  \frac{1}{\gamma_{mn}} \right) ^{k-j} }\\
& = & \sum_{\abs{\lambda_{mn}}>\kappa^{1+\epsilon}d\abs{z}}\abs{z}^2\abs{\frac{d }{\lambda_{mn}}- \frac{1}{\gamma_{mn}}}\left(\sum_{k=1}^\infty\frac{\abs{z}^{k-1}}{k+1}\abs{\sum_{j=0}^k \left( \frac{d }{\lambda_{mn}}\right)^j \left(  \frac{1}{\gamma_{mn}} \right) ^{k-j} }\right).
\end{eqnarray*}
We now focus on the inner term and estimate this directly:
\begin{eqnarray*}
\abs{\sum_{j=0}^k \left( \frac{d }{\lambda_{mn}}\right)^j \left(  \frac{1}{\gamma_{mn}} \right) ^{k-j} } & \leq & \abs{ \sum_{j=0}^k \left( \frac{d }{\lambda_{mn}}\right)^j \left(  \frac{1}{\gamma_{mn}} \right) ^{k-j} -\frac{(k+1)d^k}{\lambda_{mn}^k} } + \frac{(k+1)d^k}{\abs{\lambda_{mn}}^k}
\end{eqnarray*}
To estimate
\begin{equation*}
\abs{ \sum_{j=0}^k \left( \frac{d }{\lambda_{mn}}\right)^j \left(  \frac{1}{\gamma_{mn}} \right) ^{k-j} -\frac{(k+1)d^k}{\lambda_{mn}^k} }
\end{equation*}
we use $a=\frac{d }{\lambda_{mn}}$ , $b=  \frac{1}{\gamma_{mn}}$ and $\abs{b} \leq \kappa \abs{a}$. We have that

\begin{eqnarray*} \abs{a^{k-j}(b^j - a^j) } &=& \abs{a}^{k-j} \abs{b-a} \abs{b^{j-1} + b^{j-2} a + \cdots + a^{j-1}} \\
&\leq& \abs{a}^{k-j} \abs{b-a} (\kappa^{j-1} \abs{a}^{j-1}+ \kappa^{j-2} \abs{a}^{j-1} + \cdots + \abs{a}^{j-1}) \\
&\leq& \abs{a}^{k-1} \abs{b-a} \frac{\kappa^j}{\kappa -1}
\end{eqnarray*}
Using this, we have that
\begin{eqnarray*} \abs{ \sum_{j=0}^k \left(a^j b^{k-j} -(k+1)a^k \right) } &=& \abs{ a^{k-1}(b-a) + a^{k-2}(b^2 - a^2) + \cdots + (b^k-a^k)}\\
&\leq& \sum_{j=0}^k  \abs{a}^{k-1} \abs{b-a} \frac{\kappa^j}{\kappa -1} \\
&\leq& \abs{a}^{k-1} \abs{a-b} \frac{(k+1) \kappa ^k}{\kappa -1}\\
& \leq & \frac{(k+1) \kappa ^k}{\kappa -1}\abs{\frac{d }{\lambda_{mn}}- \frac{1}{\gamma_{mn}}}\frac{d^{k-1}}{\abs{\lambda_{mn}}^{k-1}}
\end{eqnarray*}
where in the end we recall that $a=\frac{d}{\lambda_{mn}}$ and $b=\frac{1}{\gamma_{mn}}$.  Therefore, we have
\begin{eqnarray*}
\abs{\sum_{j=0}^k \left( \frac{d }{\lambda_{mn}}\right)^j \left(  \frac{1}{\gamma_{mn}} \right) ^{k-j} } & \leq & \abs{ \sum_{j=0}^k \left( \frac{d }{\lambda_{mn}}\right)^j \left(  \frac{1}{\gamma_{mn}} \right) ^{k-j} -\frac{(k+1)d^k}{\lambda_{mn}^k} } + \frac{(k+1)d^k}{\abs{\lambda_{mn}}^k}\\
& \leq & \frac{(k+1)d^k}{\abs{\lambda_{mn}}^k}+\frac{(k+1) \kappa ^k}{\kappa -1}\frac{d^{k-1}}{\abs{\lambda_{mn}}^{k-1}}\abs{\frac{d }{\lambda_{mn}}- \frac{1}{\gamma_{mn}}}\\
& = & \frac{(k+1)d^{k-1}}{\abs{\lambda_{mn}}^{k-1}}\left(\frac{1}{\abs{\lambda_{mn}}}+\frac{\kappa^{k}}{\kappa-1}\abs{\frac{d }{\lambda_{mn}}- \frac{1}{\gamma_{mn}}}\right).
\end{eqnarray*}
Thus, we have that
\begin{eqnarray*}
\sum_{k=1}^{\infty}\frac{\abs{z}^{k-1}}{k+1}\abs{\sum_{j=0}^k \left( \frac{d }{\lambda_{mn}}\right)^j \left(  \frac{1}{\gamma_{mn}} \right) ^{k-j} }  & \leq & \sum_{k=1}^{\infty}\left(\frac{d\abs{z}}{\abs{\lambda_{mn}}}\right)^{k-1}\left(\frac{1}{\abs{\lambda_{mn}}}+\frac{\kappa^{k}}{\kappa-1}\abs{\frac{d }{\lambda_{mn}}- \frac{1}{\gamma_{mn}}}\right)\\
& \leq & \sum_{k=1}^{\infty}\left(\frac{1}{\kappa^{1+\epsilon}}\right)^{k-1}\left(\frac{1}{\abs{\lambda_{mn}}}+\frac{\kappa^{k}}{\kappa-1}\abs{\frac{d }{\lambda_{mn}}- \frac{1}{\gamma_{mn}}}\right)\\
& = & C_1(\kappa,\epsilon)\frac{1}{\abs{\lambda_{mn}}}+C_2(\kappa,\epsilon)\abs{\frac{d }{\lambda_{mn}}- \frac{1}{\gamma_{mn}}}.
\end{eqnarray*}
To arrive at this estimate, we used that $\kappa>1$ and that $\abs{\lambda_{mn}}>\kappa^{1+\epsilon}d\abs{z}$ (which appears in the sum we are considering).  Back to our estimates of $h_3(z)$, we find that
\begin{eqnarray*}
\abs{\log\abs{h_3(z)}} & \leq & \sum_{\abs{\lambda_{mn}}>\kappa^{1+\epsilon}d\abs{z}}\abs{z}^2\abs{\frac{d }{\lambda_{mn}}- \frac{1}{\gamma_{mn}}}\left(\sum_{k=1}^\infty\frac{\abs{z}^{k-1}}{k+1}\abs{\sum_{j=0}^k \left( \frac{d }{\lambda_{mn}}\right)^j \left(  \frac{1}{\gamma_{mn}} \right) ^{k-j} }\right)\\
& \leq & \sum_{\abs{\lambda_{mn}}>\kappa^{1+\epsilon}d\abs{z}} \abs{z}^2\abs{\frac{d }{\lambda_{mn}}- \frac{1}{\gamma_{mn}}}\left(C_1(\kappa,\epsilon)\frac{1}{\abs{\lambda_{mn}}}+C_2(\kappa,\epsilon)\abs{\frac{d }{\lambda_{mn}}- \frac{1}{\gamma_{mn}}}\right)\\
&=& \abs{z}^2\left(C_1(\epsilon,\kappa)\sum_{\abs{\lambda_{mn}}>\kappa^{1+\epsilon}d\abs{z}}  \abs{\frac{d }{\lambda_{mn}}- \frac{1}{\gamma_{mn}} }^2 
+ C_2(\epsilon,\kappa) \sum_{\abs{\lambda_{mn}}>\kappa^{1+\epsilon}d\abs{z}} \frac{1}{\abs{\lambda_{mn}}} \abs{\frac{d }{\lambda_{mn}}- \frac{1}{\gamma_{mn}} }\right).
\end{eqnarray*}
The two series above are remainders of convergent series due to part (b) of Lemma~\ref{dreg}. So they both tend to 0, as $\abs{z} \rightarrow \infty$. Therefore, the last expression is $o(\abs{z}^2)$. To conclude, we have
$$
\abs{\log\abs{h_3(z)}} \leq o(\abs{z}^2)
$$
or equivalently,
$$
e^{-o(\abs{z}^2)}\leq\abs{h_3(z)}\leq e^{o(\abs{z}^2)}
$$
which is \eqref{h3est}.

We finally turn to term $h_2(z)$,
\begin{equation*} 
 h_2(z)  =  \frac{\textnormal{dist}(d z,\Lambda)}{\textnormal{dist}(z,\Gamma)}\frac{(z-\gamma_{00})}{d z}\sideset{}{'}{\prod}_{\abs{\lambda_{mn}}\leq \kappa^{1+\epsilon}d\abs{z}}\frac{\left(1-\frac{z}{\gamma_{mn}}\right)}{\left(1-\frac{d z}{\lambda_{mn}}\right)}.
\end{equation*}

Let $\frac{1}{d}\lambda_{MN}$ and $\gamma_{M'N'}$ be the closest points from $\frac{1}{d}\Lambda$ and $\Gamma$ to $z$ respectively. Notice that they both depend on $z$.  We can assume that $(0,0)\neq(M,N)\neq (M',N')\neq(0,0)$, other cases being similar and easier.  Then we have
\begin{eqnarray}
 \label{logh2} \log{\abs{h_2(z)}}&=&\log|z-\gamma_{00}|-\log|d z|-\log|\gamma_{M'N'}|+\log{\abs{\frac{\lambda_{MN}}{z}}} \notag \\ 
&+&{{\sum}'_{|\lambda_{mn}|\leq  d\kappa^{1+\epsilon}\abs{z}}}\log\abs{1-\frac{z}{\gamma_{mn}}}- {{\sum}'_{|\lambda_{mn}|\leq  d\kappa^{1+\epsilon}\abs{z}}}\log\abs{1-\frac{d z}{\lambda_{mn}}}.
\end{eqnarray}
The primes in the sums denote that they are possibly missing the terms $\log\abs{1-z/\gamma_{M'N'}}$ and $\log\abs{1-d z/\lambda_{MN}}$ respectively. Notice that the first four terms  in~\eqref{logh2} are bounded by $O(\log|z|)$. The rest of \eqref{logh2} is equal to
\begin{eqnarray} \label{h2rest}
&& \sump_{|\lambda_{mn}|\leq  d\kappa^{1+\epsilon}\abs{z}}\log\abs{1-\frac{z}{\gamma_{mn}}}- \sump_{|\lambda_{mn}|\leq  d\kappa^{1+\epsilon}\abs{z}}\log\abs{1-\frac{d z}{\lambda_{mn}}} \notag \\ 
&=& \sump_{|\lambda_{mn}|\leq  d\kappa^{1+\epsilon}\abs{z}}\log\abs{1-\frac{\lambda_{mn}/d-\gamma_{mn}}{\lambda_{mn}/d-z}}- 
 \sump_{|\lambda_{mn}|\leq  d\kappa^{1+\epsilon}\abs{z}}\log\abs{1-\frac{\lambda_{mn}/d-\gamma_{mn}}{\lambda_{mn}/d}}. 
 \end{eqnarray}
First we estimate the first sum in~\eqref{h2rest}.
\[ \sideset{}{'}{\sum}_{|\lambda_{mn}|\leq  d\kappa^{1+\epsilon}\abs{z}}\log\abs{1-\frac{\lambda_{mn}/d-\gamma_{mn}}{\lambda_{mn}/d-z}}\leq  \sideset{}{'}{\sum}_{|\lambda_{mn}|\leq  d\kappa^{1+\epsilon}\abs{z}} \abs{\frac{\lambda_{mn}-d\gamma_{mn}}{\lambda_{mn}-d z}}\]
\[\leq 2 \sideset{}{'}{\sum}_{|\lambda_{mn}|\leq  d\kappa^{1+\epsilon}\abs{z}} \abs{\frac{\lambda_{mn}-d\gamma_{mn}}{\lambda_{mn}-\lambda_{MN}}}
\leq  2 \sideset{}{'}{\sum}_{|\lambda_{mn}|\leq  d\kappa^{1+\epsilon}\abs{z}} \abs{\frac{-\lambda_{MN}+d\gamma_{mn}+\lambda_{M-m,N-n}}{\lambda_{M-m,N-n}}}\]

\begin{equation} \label{problemsum} \leq  2 \sideset{}{'}{\sum}_{|\lambda_{mn}|\leq  d\kappa^{1+\epsilon}\abs{z}} \abs{\frac{-d\gamma_{M-m,N-n}+\lambda_{M-m,N-n}}{\lambda_{M-m,N-n}}}+ 2 \sideset{}{'}{\sum}_{|\lambda_{mn}|\leq  d\kappa^{1+\epsilon}\abs{z}} \abs{\frac{-\lambda_{MN}+d\gamma_{mn}+d\gamma_{M-m,N-n}}{\lambda_{M-m,N-n}}}.\end{equation}

Now, using that $|\lambda_{mn}-d\gamma_{mn}|\leq \phi(|\lambda_{mn}|)$ we can estimate the second sum in \eqref{problemsum} the following way:

\begin{eqnarray*}\sideset{}{'}{\sum}_{|\lambda_{mn}|\leq  d\kappa^{1+\epsilon}\abs{z}} \abs{\frac{d\gamma_{mn}-\lambda_{MN}+d\gamma_{M-m,N-n}}{\lambda_{M-m,N-n}}}&\leq& 
C\sideset{}{'}{\sum}_{|\lambda_{mn}|\leq  d\kappa^{1+\epsilon}\abs{z}} \frac{\phi(|\lambda_{mn}|)+\phi(|\lambda_{M-m,N-n}|)}{|\lambda_{M-m,N-n}|}\\
&\leq&
%C\sum_{|\lambda_{mn}|\leq  d\kappa^{1+\epsilon}\abs{z}} {\frac{\phi(|\lambda_{M-m,N-n}|)}{|\lambda_{M-m,N-n}|}}\]
C\sideset{}{'}{\sum}_{|\lambda_{mn}|\leq  d\kappa^{1+\epsilon}\abs{z}} {\frac{ 2\kappa^{1+\epsilon} d \abs{z}(\phi(\kappa^{1+\epsilon} d \abs{z})  +\phi(2 \kappa^{1+\epsilon} d \abs{z}))}{|\lambda_{M-m,N-n}|^2}} \\
%C\sum_{|\lambda_{mn}|\leq  d\kappa^{1+\epsilon}\abs{z}} {\frac{2 \kappa^{1+\epsilon} d \abs{z}\phi(2 \kappa^{1+\epsilon} d \abs{z})}{|\lambda_{M-m,N-n}|^2}}
&=& C'  2\kappa^{1+\epsilon} d \abs{z}\phi(2\kappa^{1+\epsilon} d \abs{z}) \sideset{}{'}{\sum}_{|\lambda_{mn}|\leq  d\kappa^{1+\epsilon}\abs{z}} {\frac{ 1 }{|\lambda_{M-m,N-n}|^2}} \\
&\leq& C'  2\kappa^{1+\epsilon} d \abs{z}\phi(2\kappa^{1+\epsilon} d \abs{z}) \sideset{}{'}{\sum}_{|\lambda_{mn}|\leq  2d\kappa^{1+\epsilon}\abs{z}} {\frac{ 1 }{|\lambda_{m,n}|^2}} \\
&\leq& C'' 2\kappa^{1+\epsilon} d \abs{z}\phi(2\kappa^{1+\epsilon} d \abs{z}) \log{(2\kappa^{1+\epsilon} d \abs{z}) } = o(\abs{z}^2).
 \end{eqnarray*}
In the second inequality above we used that $ \abs{\lambda_{M-m,N-n}} \leq 2 \kappa^{1+\epsilon} d \abs{z}$ and that $\phi(t)$ is increasing. In the last equality we used Lemma~\ref{phi_est}.

Next we estimate the first sum in \eqref{problemsum}. Using part (b) from Lemma~\ref{dreg}

\begin{eqnarray*} \sideset{}{'}{\sum}_{|\lambda_{mn}|\leq  d\kappa^{1+\epsilon}\abs{z}} \abs{\frac{-d\gamma_{M-m,N-n}+\lambda_{M-m,N-n}}{\lambda_{M-m,N-n}}} 
&\leq& 
\sideset{}{'}{\sum}_{|\lambda_{mn}|\leq  2 d\kappa^{1+\epsilon}\abs{z}} \abs{\frac{-d\gamma_{mn}+\lambda_{mn}}{\lambda_{mn}}} \\
&=& \sideset{}{'}{\sum}_{|\lambda_{mn}|\leq  2 d\kappa^{1+\epsilon}\abs{z}}\abs{\gamma_{mn}}\abs{\frac{1}{\gamma_{mn}} - \frac{d}{\lambda_{mn}}} \leq o(\abs{z^2}).
\end{eqnarray*}  
This shows that the first sum in \eqref{h2rest} is bounded from above by $o(\abs{z}^2)$. Next we bound the second sum in \eqref{h2rest}.
\begin{eqnarray*} \sideset{}{'}{\sum}_{|\lambda_{mn}|\leq  d\kappa^{1+\epsilon}\abs{z}}\log\abs{1-\frac{\lambda_{mn}/d-\gamma_{mn}}{\lambda_{mn}/d}} 
&=& -\sideset{}{'}{\sum}_{|\lambda_{mn}|\leq  d\kappa^{1+\epsilon}\abs{z}}\log\abs{1-\frac{\gamma_{mn}-\lambda_{mn}/d}{\gamma_{mn}}} \\
&\geq& -\sideset{}{'}{\sum}_{|\lambda_{mn}|\leq  d\kappa^{1+\epsilon}\abs{z}} \frac{\abs{\lambda_{mn}}}{d} \abs{ \frac{d}{\lambda_{mn}} - \frac{1}{\gamma_{mn}}} \\
&\geq& -\kappa^{1+\epsilon} \abs{z} \sideset{}{'}{\sum}_{|\lambda_{mn}|\leq  d\kappa^{1+\epsilon}\abs{z}} \abs{ \frac{d}{\lambda_{mn}} - \frac{1}{\gamma_{mn}}} \\
&\geq& -o(\abs{z}^2).
\end{eqnarray*}
This shows that $\log{\abs{h_2(z)}} \leq o(\abs{z}^2)$. 

Finally, we bound $\log{\abs{h_2(z)}}$ from below. For the second sum in~\eqref{h2rest} we have:
\begin{eqnarray*} \sideset{}{'}{\sum}_{|\lambda_{mn}|\leq  d\kappa^{1+\epsilon}\abs{z}}\log\abs{1-\frac{\lambda_{mn}/d-\gamma_{mn}}{\lambda_{mn}/d}} 
&\leq& \sideset{}{'}{\sum}_{|\lambda_{mn}|\leq  d\kappa^{1+\epsilon}\abs{z}} \abs{\frac{\lambda_{mn}/d-\gamma_{mn}}{\lambda_{mn}/d}}  \\
&=&  \sideset{}{'}{\sum}_{|\lambda_{mn}|\leq  d\kappa^{1+\epsilon}\abs{z}} \abs{\gamma_{mn}}  \abs{ \frac{d}{\lambda_{mn}} - \frac{1}{\gamma_{mn}}}  \\
&\leq& C \abs{z} \sideset{}{'}{\sum}_{|\lambda_{mn}|\leq  d\kappa^{1+\epsilon}\abs{z}} \abs{ \frac{d}{\lambda_{mn}} - \frac{1}{\gamma_{mn}}} \\
&\leq& o(\abs{z}^2).
\end{eqnarray*}
Lastly, we bound the first sum in~\eqref{h2rest} from below.

\begin{eqnarray*} - \sideset{}{'}{\sum}_{|\lambda_{mn}|\leq  d\kappa^{1+\epsilon}\abs{z}}\log\abs{1-\frac{\lambda_{mn}/d-\gamma_{mn}}{\lambda_{mn}/d-z}}
&=&  \sideset{}{'}{\sum}_{|\lambda_{mn}|\leq  d\kappa^{1+\epsilon}\abs{z}}\log\abs{1-\frac{\gamma_{mn}-\lambda_{mn}/d}{\gamma_{mn}-z}} \\
&\leq&  \sideset{}{'}{\sum}_{|\lambda_{mn}|\leq  d\kappa^{1+\epsilon}\abs{z}} \abs{\frac{\lambda_{mn}-d\gamma_{mn}}{d\gamma_{mn}-d z}}\\
&\leq& \frac{2}{d} \sideset{}{'}{\sum}_{|\lambda_{mn}|\leq  d\kappa^{1+\epsilon}\abs{z}} \abs{\frac{\lambda_{mn}-d\gamma_{mn}}{\gamma_{mn}-\gamma_{M'N'}}}\\
&\leq&  \frac{2}{d c} \sideset{}{'}{\sum}_{|\lambda_{mn}|\leq  d\kappa^{1+\epsilon}\abs{z}} \abs{\frac{-\lambda_{M'N'}+d\gamma_{mn}+\lambda_{M'-m,N'-n}}{\lambda_{M'-m,N'-n}}}
\end{eqnarray*}
where in the last inequality we used  that $\abs{\gamma_{mn}-\gamma_{m'n'}} \geq c \abs{\lambda_{mn}-\lambda_{m'n'}}$ for all pairs of indices. To finish this, we notice that $\abs{\lambda_{M'N'}} = O(\abs{z})$, so we can bound the last sum by  a sum over a larger disk of radius $O(\abs{z})$.  Then, continuing the estimate exactly as in \eqref{problemsum}, we obtain the desired bound $o(\abs{z}^2)$.
\end{proof}

\section{Proof of Theorem \ref{mainfornow}}
With the necessary preliminaries out of the way, we can prove the main result of this note.  For simplicity we restate the main result from the introduction.

\begin{thm}
Let $\G$ be a $d-$regular separated sequence.
\begin{itemize}
\item[(a)] If $d<1$ then $\G$ is not a uniqueness set for the Bargmann-Fock space $\FF$.
\item[(b)] If $d>1$ then $\G$ is a uniqueness set for the Bargmann-Fock space $\FF$.
\end{itemize}  
\end{thm}

\begin{proof} Let $d>1$ and let $f(z)\in \FF$ be an entire function vanishing identically on $\G$. Then $f(z)=\sG(z)g(z)$, with $g(z)$ entire. We have 
$$\int_{\C} |f(z)|^2 e^{-\pi|z|^2}dA(z)=\int_{\C} |g(z)|^2 |\sG(z)|^2e^{-\pi|z|^2}dA(z).$$ By the second property of $d$-regular sequences, there exists a sequence of non-overlapping discs with uniform radius $r>0$ which are centered at the points in $\G$. Let $U$ be the region in $\C$ obtained by deleting all these discs. Using the left inequality in the Lemma~\ref{mainlemma} we have
   
$$\int_{\C} |f(z)|^2 e^{-\pi|z|^2}dA(z)\geq C_1\int_{U} |g(z)|^2 e^{\pi(d^2-1)|z|^2-o(|z|^2)}dA(z).$$ Finally, using the subharmonicity of $|g(z+\gmn)|^2$ we obtain, 
$$ \infty>\int_{\C} |f(z)|^2 e^{-\pi|z|^2}dA(z)= c_2\int_{\C}|g(z)|^2dA(z),$$ for some constant $c_2>0$. This implies that $g(z)$, and consequently $f(z)$, is identically zero. Therefore, $\G$ must be a uniqueness set. 

If $d<1$ then the right inequality in Lemma~\ref{mainlemma} implies 
$$\int_{\C} |\sG(z)|^2 e^{-\pi|z|^2}dA(z)\leq C_2\int_{\C}  e^{-\pi(1-d^2)|z|^2+o(|z|^2)}dA(z)<\infty.$$ Thus, $\sG(z)\in \FF$ and $\G$ is not a uniqueness set.
\end{proof}

\begin{rem} It is easy to see that similar results continue to hold in spaces $\FF_p$ of entire functions $F(z)$ satisfying $$\int_{\C}|F(z)|^pe^{-\frac{p\pi|z|^2}{2}}dA(z)<\infty.$$  
\end{rem}

\section{Concluding Remarks}

An interesting open question would be to obtain a complete geometric characterization of the uniqueness sets for the Fock space $\mathcal{F}$.  Based on the results in this paper, and the results of \cite{ALS}, a resolution of this question will be quite subtle and challenging.

In a forthcoming paper, we hope to provide similar results about uniqueness sets and, consequently, zero sets in the Bergman space.

%%%%%%%%%%%%
%%%References%%%
%%%%%%%%%%%%

%%%%%%%%%%%%
%%%%END%%%%%%
%%%%%%%%%%%%

\end{document}